\documentclass[11pt,letterpaper]{amsart}
\usepackage{amsmath, amsthm, amsfonts, amssymb}
\usepackage{enumerate}
\usepackage[bookmarks=false]{hyperref}

\usepackage{xcolor}

\usepackage{verbatim}

\title{Structure of relatively bi-exact group von Neumann algebras}

\author{Changying Ding}
\address{Department of Mathematics, Vanderbilt University, 1326 Stevenson Center, Nashville, TN 37240, USA}
\email{changying.ding@vanderbilt.edu}

\author{Srivatsav Kunnawalkam Elayavalli}
\address{Institute of Pure and Applied Mathematics, 460 Portola Plaza, Los Angeles, CA 90095}
\email{srivatsav.kunnawalkam.elayavalli@vanderbilt.edu}

\newtheorem{thm}{Theorem}[section]
\newtheorem{prop}[thm]{Proposition}
\newtheorem{cor}[thm]{Corollary}
\newtheorem{lem}[thm]{Lemma}
\theoremstyle{definition}
\newtheorem{defn}[thm]{Definition}
\newtheorem{defn/lem}[thm]{Definition/Lemma}
\newtheorem{rem}[thm]{Remark}
\newtheorem{examp}[thm]{Example}

\newtheorem{problem}{Problem}[]

\newtheorem{conj}[problem]{Conjecture}

\newcommand{\C}{{\mathbb C}}
\newcommand{\bB}{{\mathbb B}}

\newcommand{\K}{{\mathbb K}}
\newcommand{\M}{{\mathbb M}}

\newcommand{\bS}{{\mathbb S}}

\newcommand{\Y}{{\mathbb Y}}

\newcommand{\X}{{\mathbb X}}

\newcommand{\PP}{{\mathcal P}}

\newcommand{\B}{{\mathbb B}}

\newcommand{\cH}{{\mathcal H}}

\newcommand{\cU}{{\mathcal U}}

\newcommand{\cZ}{{\mathcal Z}}

\newcommand{\Ad}{\operatorname{Ad}}

\newcommand{\id}{\operatorname{id}}

\newcommand{\ds}{{\sharp\kern-.5pt\sharp}}

\newcommand{\actson}{{\, \curvearrowright \,}}

\newcommand{\RN}[1]{%
  \textup{\uppercase\expandafter{\romannumeral#1}}%
}

\makeatletter
\DeclareRobustCommand\frownotimes{\mathbin{\mathpalette\frown@otimes\relax}}
\newcommand{\frown@otimes}[2]{%
  \vbox{
    \ialign{##\cr
      \hidewidth$\m@th#1{}_\frown$\kern-\scriptspace\hidewidth\cr
      \noalign{\nointerlineskip\kern-1pt}
      $\m@th#1\otimes$\cr
    }%
  }%
}
\makeatother

\begin{document}
\begin{abstract}

Using computations in the bidual of $\mathbb{B}(L^2M)$ we develop a new technique at the von Neumann algebra level to upgrade relative proper proximality to full proper proximality.   
 This is used to structurally classify subalgebras of $L\Gamma$ where $\Gamma$ is an infinite group that is biexact relative to a finite family of subgroups $\{\Lambda_i\}_{i\in I}$ such that each $\Lambda_i$ is almost malnormal in $\Gamma$. This generalizes the result of \cite{DKEP21} which classifies subalgebras of von Neumann algebras of biexact groups.     By developing a combination  with techniques from Popa's deformation-rigidity theory  we obtain a new structural absorption theorem for free products and a generalized Kurosh type theorem in the setting of properly proximal von Neumann algebras.  
\end{abstract}

\maketitle

\section{Introduction }

Recently the authors and J. Peterson in \cite{DKEP21}  developed the theory of small at infinity compactifications  a la Ozawa (\cite{BO08}), in the setting of tracial von Neumann algebras. At the foundation of this work lies the theory of operator $M$-bimodules and the several natural topologies that arise in this setting (see \cite{EfRu88}, \cite{Ma97,Ma98,Ma05}). The small at infinity compactification is  a canonical strong operator bimodule (in the sense of Magajna \cite{Ma97}) containing the compact operators. By using the noncommutative Grothendieck inequality (similar to Ozawa in \cite{Oza10}) it was seen that this strong operator bimodule coincides with $\K^{\infty,1}(M)$,  the closure $\mathbb K^{\infty, 1}(M)$ of $\mathbb K(L^2M)$ with respect to the $\|\cdot\|_{\infty, 1}$-norm on $\mathbb B(L^2M)$ given by $\| T \|_{\infty, 1} = \sup_{x \in M, \| x \| \leq 1} \| T \hat{x} \|_1$. The small at infinity compactification of a tracial von Neumann algebra $M$ is then given by \[
\mathbb S(M) = \{ T \in \mathbb B(L^2M) \mid [T, JxJ] \in \mathbb K^{\infty, 1}(M), \ {\rm for\ all\ }x \in M \}. 
\]
It is easy to see that this operator $M$-system $\mathbb S(M)$ contains $M$ and   $\mathbb K(L^2M)$, and is an $M$-bimodule. The    advantage of the  strong operator bimodule perspective is that it to identify an operator $T\in \mathbb S(M)$ suffices to check that $[T, JxJ] \in \mathbb K^{\infty, 1}(M)$ for all $x$ in some weakly closed subset of $M$. This is what allows for the  passage between the group and the von Neumann algebra settings. Using this technology \cite{DKEP21} defined the notion of proper proximality for finite von Neumann algebras, extending the dynamical notion for groups \cite{BIP18}: A finite von Neumann algebra $(M,\tau)$ is  properly proximal if there does not exist an $M$-central state $\varphi$ on  $\mathbb S(M)$ such that $\varphi_{|M}=\tau$.  By identifying and studying this property in various examples, the authors of \cite{DKEP21} obtained applications to the structure theory of II$_1$-factors. The goal of the present paper is to add to the list of applications.  

The  machinery underlying the results in this paper  is built on  is the notion of an $M$-boundary piece developed in \cite{DKEP21}, as an analogue of the group theoretic notion introduced in \cite{BIP18}. The motivation for considering this notion is that it allows for one to exploit the dynamics that is available only on certain locations of the  Stone-Cech boundary of the group.    For a group $\Gamma$, a boundary piece is a closed left and right invariant subset of $\beta(\Gamma)\setminus \Gamma $, whereas in the von Neumann algebra setting, it is denoted by $\X$ typically and is a certain hereditary $C^*$-subalgebra of $\mathbb B(L^2M)$ containing the compact operators (see Section 3.1).  One then considers the small at infinity compactification relative to a boundary piece $\mathbb S_{\X}(M)$ where $\mathbb K^{\infty, 1}(M)$ is replaced with $\mathbb K_{\X}^{\infty, 1}(M)$, a suitable analogue for the boundary piece. Then one  can define the notion of  proper proximality relative to $\X$, demanding that there be no $M$-central state restricting to the trace on $\mathbb S_{\X}(M)$.  The main example we will be working with is a boundary piece generated by a finite family of von Neumann subalgebras $\{M_i\}_{i=1}^{n}$ (see Example \ref{examp:bdrysub}), which is adapted from the construction for a finite family of subgroups (see Example 3.3 in \cite{BIP18}).

 In \cite{DKE21}, the authors demonstrated an instance where   relative  proper proximality can be lifted to full proper proximality, i.e, when the boundary piece arises from subgroups that are almost malnormal \footnote{A subgroup $H<G$ is almost malnormal if for all $g\in G\setminus H$, $gHg^{-1}\cap H$ is finite.}  and not co-amenable (see Lemma 3.3 in \cite{DKE21}). The authors used this idea to classify proper proximality for wreath product groups. In this paper, we develop an analogue of this idea in the setting of von Neumann algebras (Theorem \ref{prop:bootstrap}). In both cases, one has to work in the bidual of the small at infinity compactification for technical reasons, and this brings about an extra layer of subtlety especially in the von Neumann setting. More specifically we show that one can map the basic construction into the bidual version of the relative small at infinity  compactification, provided the boundary piece arises from a mixing subalgebra. Composing with an appropriate state on this space, we get the link with relative amenability in the von Neumann setting. This upgrading theorem is the main new technical tool we develop in the present work:

\begin{thm}\label{prop:bootstrap}
Let $M$ be a diffuse finite von Neumann algebra, $M_i\subset M$, $i=1,\dots ,n$ diffuse von Neumann subalgebras
    such that the Jones projections $e_{M_i}$ pairwise commute, $M_i\subset M$ admits a {bounded} Pimsner-Popa basis (see Definition \ref{bpp}),
    and $A\subset pMp$ is a von Neumann subalgebra, for some $p\in \mathcal P(M)$.
Suppose that $A$ is properly proximal relative to $\X$ inside $M$, where $\X$ is the boundary piece associated with $\{M_i\}_{i=1}^n$,
    and $M_i\subset M$ is mixing for each $i=1,\dots, n$.
Then there exist projections $f_0\in \mathcal Z(A)$ and $f_i \in \cZ(A'\cap pMp)$, $1\leq i\leq n$,
    such that $Af_0$ is properly proximal and $Af_i$ is amenable relative to $M_i$ inside $M$ for each $1\leq i\leq n$,
    and $\sum_{i=0}^n f_i =p$.
\end{thm}


Using these ideas we  are interested in classifying  subalgebras of group von Neumann algebras arising from groups that are biexact relative to a family of subgroups (see e.g. \cite[Chapter 15]{BO08}). The first result of this kind was obtained in Theorem 7.2 of \cite{DKEP21} where it was shown that every subalgebra of the von Neumann algebra of a biexact group either has an amenable direct summand or is properly proximal. As essentially observed there, what relative biexactness buys us is the relative proper proximality for any subalgebra, relative to the boundary piece arising from the subgroups. Combining this with our upgrading result above,  we obtain our     main result below which is a structure theorem for  von Neumann subalgebras of group von Neumann algebras that are biexact relative to a family of subgroups where each subgroup is almost malnormal.

\begin{thm}\label{cor:biexact}
Let $\Gamma$ be a countable  group with a family of almost malnormal subgroups $\{\Lambda_i\}_{i=1}^n$.
If $\Gamma$ is biexact relative to  $\{\Lambda_i\}_{i=1}^n$,  
then for any von Neumann subalgebra $A\subset L\Gamma$, there exists $p\in \cZ{(A)}$ and projections $p_j\in \mathcal Z((Ap)'\cap pL(\Gamma)p)$ such that $\bigvee_{j=1}^{n} p_j=p$ and $Ap_i$ is amenable relative to $L\Lambda_i$ inside $L\Gamma$, for each $i=1,\dots,n$ and $Ap^\perp$ is properly proximal.
\end{thm}


There are two natural instances where such a phenomenon (a countable group $\Gamma$ with a family of almost malnormal subgroups $\{\Lambda_i\}_{i=1}^n$
where $\Gamma$ is biexact relative to  $\{\Lambda_i\}_{i=1}^n$) is observed: First is in  the setting of free products, which we deal with in the present paper. Second is in  the setting of wreath products, which is investigated in a follow-up work by the first author \cite{Din22}. There is conjecturally a third setting of relative hyperbolicity, which we comment on in the end of the introduction.    

Thanks to Bass-Serre theory \cite{Ser80} we have a complete understanding of subgroups of a free product of groups. As a result, one can derive results of the following nature: If $H<G_1*G_2$ such that $|H\cap G_1|\geq 3$, then $H$ is amenable  only if $H<G_1$.
This phenomenon is referred to as amenable absorption.  Interestingly, the situation for von Neumann algebras is much more complicated. There is comparatively a very limited understanding of von Neumann subalgebras of free products. 
 Whether every self adjoint operator in any finite von Neumann algebra is contained in a copy of the hyperfinite II$_1$-factor  was itself an open problem for many years\footnote{This is a question of Kadison, Problem 7 from 'Problems on von Neumann algebras, Baton Rouge Conference' }. 
 Popa settled it in the negative in \cite{Popa1983MaximalIS} by discovering a surprising amenable absorption theorem for free product von Neumann algebras, thereby showing that  a generator masa in $L\mathbb{F}_2$ is maximally amenable. 

Popa's ideas been used to show maximal amenability in other situations (See for instance \cite{radialmasa}, \cite{wen}, \cite{SKC}, \cite{1c}). In the past decade there have been other new ideas that have been used to prove absorption theorems: Boutonnet-Carderi's approach \cite{BC15} relies on elementary computations in a crossed-product $C^*$-algebra; Boutonnet-Houdayer \cite{BH18} use the study of non normal conditional expectations; \cite{hayesetal} used a free probabilistic approach to study absorption. Ozawa in \cite{ozawanote} then gave a short proof of amenable absorption in tracial free products. There have also been a variety of important free product absorption results which are of a different flavor, and are structural in nature. See for example \cite{IPP08} and  \cite{ChHou}.     

  By applying our Theorem \ref{cor:biexact} in the setting of free products and using machinery from Popa's deformation-rigidity theory (specifically work of Ioana \cite{Io18}),
we obtain  a generalized structural absorption theorem below:

\begin{cor}\label{cor:inner amen absorption}
Let $(M_1,\tau_1)$ and $(M_2, \tau_2)$ be such that $M_i\cong L\Gamma_i$ where $\Gamma_i$ are countable exact groups and $M=M_1\ast M_2$ be the tracial free product.
Let $A\subset M$ be a von Neumann subalgebra with $A\cap M_1$ diffuse. 
If $A\subset M$  has no properly proximal direct summand, then $A\subset M_1$.
\end{cor}

\begin{rem}
    Using results of the upcoming work \cite{DP22}, one can relax the assumption on $M_i$, from being infinite group von Neumann algebras of exact groups, to just that they are  diffuse weakly exact von Neumann algebras.  We do not comment more on this at the moment because for the sake of examples, the above setting already provides many.     
\end{rem}

The authors of \cite{IsPeRu19} showed that there are examples of groups that are neither inner amenable nor properly proximal. All of these group von Neumann algebras fit into the setting of the above corollary.  
  Note that Vaes constructed in \cite{vaesinneramenable} plenty of groups that are inner amenable, yet their group von Neumann algebras lack Property (Gamma). Hence our results give a strict generalization of (Gamma) absorption (see Houdayer's Theorem 4.1 in \cite{houdayergammaabsorption} and see also Theorem A in \cite{hayesetal}) in these examples. 

\begin{rem}
The above result is false if one considers amalgamated free products. For instance, take $M_1=M_2=L\mathbb F_2\overline \otimes R$ and $A=(L\mathbb Z\ast L\mathbb Z)\overline\otimes R\subset M_1\ast_R M_2$, where two copies of $L\mathbb Z$ in $A$ are from $M_1$ and $M_2$, respectively. 
\end{rem}

\begin{rem}
Shortly before the posting of this paper, Drimbe announced a paper (see \cite{Drimbe1}) where he shows using Popa's deformation-rigidity theory that for any nonamenable inner amenable group $\Gamma$, if $L(\Gamma)\subset M_1*M_2$, then $L(\Gamma)$ intertwines into  $M_i$ for some $i=1,2$. This in particular generalizes  Corollary \ref{cor:inner amen absorption} in the case that $A\cong L\Gamma$ for some inner amenable group $\Gamma$, because he doesn't require  any assumptions for $M_i$.   
\end{rem}

Our techniques also reveal the following new Kurosh type structure theorem for free products in the setting of proper proximality, (partially generalizing Corollary 8.1 in \cite{Drimbe1}). See also \cite{Oza06, IPP08, Jesse, Houdayer_2016} for other important Kurosh type theorems.

\begin{cor}\label{Kurosh}
    Let $M=L\Gamma_1* \cdots *L\Gamma_m=  L\Lambda_1* \cdots *L\Lambda_n$, where all groups $\Gamma_i$ and $\Lambda_j$ 
	are countable exact nonamenable non-properly proximal i.c.c.\ groups. Then $m=n$ and  after a permutation of indices $L\Gamma_i$ is unitarily conjugate to $L\Lambda_i$.  
\end{cor}


We conclude by state the following folklore conjecture (also stated in \cite{Koichi}), which would provide another family of examples for applying Theorem \ref{cor:biexact}. Indeed the peripheral subgroups below are almost malnormal (see Theorem 1.4 in \cite{osin}). 

\begin{conj}[\cite{Koichi}]
If $G$ is exact and hyperbolic relative to a family of peripheral subgroups $\{H_i\}_{i=1}^{n}$, then $G$ is biexact relative to $\{H_i\}_{i=1}^{n}$.  
\end{conj}

\subsection*{Acknowledgements} The authors thank J. Peterson for stimulating conversations and very helpful suggestions. The authors thank Ben Hayes  for reading our  very early draft and offering many comments.

\section{Preliminaries}

\subsection{The basic construction and Pimsner-Popa orthogonal bases}\label{PP}

Let $M$ be a finite von Neumann algebra and $Q\subset M$ be a von Neumann subalgebra. The  basic construction $\langle M,e_Q\rangle$ is defined as the von Neumann subalgebra of $\mathbb B(L^2M)$ generated by $M$ and the orthogonal projection $e_Q$ from $L^2(M)$ onto $L^2(Q)$. 
There is a semifinite faithful normal trace on $\langle M,e_Q\rangle$ satisfying $\text{Tr}(xe_Q y)=\tau(xy)$, for every $x,y\in M$. 

Let $N\subset M$ be a von Neumann subalgebra. Then a Pimsner-Popa basis  (see \cite{PPbasis}) of $M$ over $N$ is a family of elements denoted $M/N=\{m_j\}_{j\in J}\subset M$ such that \begin{enumerate}
    \item ${E}_{N}(m_j^*m_k)=\delta_{j,k}p_j$, where $p_j\in\mathcal P(N)$ is a projection.
    \item \label{PPbasisfact2}$L^2(M)=\bigoplus_{j\in J}m_jL^2(N)$ and every $x\in M$ has a unique decomposition $x=\sum_{j} m_j{E}_N(m_j^*x)$. 
\end{enumerate}

In the case that $N=L(\Lambda)$ and $M=L(\Gamma)$ where $\Lambda<\Gamma$, we can identify a  Pimnser-Popa basis in $M$ from a choice of coset-representatives i.e, $\Gamma = \bigsqcup_{k\geq 0} t_k\Lambda$,
	and $m_k:=\lambda_{t_k}\in \cU(L\Gamma)$: $M/N=\{u_j\}_{j\in J}$.


    

For technical reasons, we will need the existence of the following type of Pimsner-Popa basis for our results:

\begin{defn}\label{bpp}
    Say that an inclusion of separable finite von Neumann algebras $N\subset M$ admits a \emph{bounded} Pimnser-Popa basis in $M$ if there exists a Pimsner-Popa basis $\{m_k\}_{k\in \mathbb{N}}$ for the inclusion $N\subset M$ such that $\sup_{k\in \mathbb{N}}\|m_k\|<\infty$. 
\end{defn}

Note that if $M/N$ consists of unitaries in $M$, then it clearly also satisfies that it is a {bounded} Pimnser-Popa basis. This is a technical property considered by Ceccherini-Silberstein \cite{ceccherini}, called the U-property. It is a well known open problem if such bases always exist.   

Such a Pimnser-Popa basis satisfying the above U-property always exists for the inclusion $L\Lambda \subset L\Gamma$ where  $\Lambda<\Gamma$ is a subgroup of a countable group $\Gamma$.  




\subsection {Popa's intertwining-by-bimodules}

\begin{thm}[\cite{popainvent}]\label{popa intertwining} Let $(M,\tau)$ be a tracial von Neumann algebra and $P\subset pM p,Q\subset M$ be von Neumann subalgebras. 
Then the following  are equivalent:

\begin{enumerate}
\item There exist projections $p_0\in P, q_0\in Q$, a $*$-homomorphism $\theta:p_0P p_0\rightarrow q_0Q q_0$  and a non-zero partial isometry $v\in q_0M p_0$ such that $\theta(x)v=vx$, for all $x\in p_0P p_0$.


\item There is no sequence $u_n\in\mathcal U(P)$ satisfying $\|E_Q(x^*u_ny)\|_2\rightarrow 0$, for all $x,y\in pM$.
\end{enumerate}

\end{thm}

If one of these equivalent conditions holds,  we write $P\prec_{M}Q$, and say that \emph{a corner of $P$ embeds into $Q$ inside $M$.}

\subsection{Relative amenability}


Let $P\subset M $ and $Q \subset M$ be a von Neumann subalgebras.
We say that $P$ is  {\it amenable relative to $Q$ inside $M$} if there exists a sequence $\xi_n\in L^2(\langle M,e_Q\rangle)$ such that $\langle x\xi_n,\xi_n\rangle\rightarrow\tau(x)$, for every $x\in M $, and $\|y\xi_n-\xi_ny\|_2\rightarrow 0$, for every $y\in P$. By \cite{OzPo10a}, Theorem 2.1  $P$ is  amenable relative to $Q$ inside $M$ if and only if there exists a $P$-central state in the basic construction $\langle M,e_Q\rangle$ that is normal when restricted to $M$, and faithful on $\mathcal{Z}(P'\cap M)$. 


\subsection{Mixing subalgebras and free products of finite von Neumann algebras.}\label{sec:free products}


Let $M$ be a finite von Neumann algebra and $N\subset M$ a von Neumann subalgebra. Recall the inclusion $N\subset M$ is mixing if $L^2(M\ominus N)$ is mixing as an $N$-$N$ bimodule, i.e., for any sequence $u_n\in\mathcal U(N)$ converging to $0$ weakly, one has $\|E_N(xu_ny)\|_2\to 0$ for any $x,y\in M\ominus N$.
When $M$ and $N$ are both diffuse, we may replace sequence of unitaries with any sequence in $N$ converging to $0$ weakly \cite[Theorem 5.9]{DKEP21}.

\begin{rem}\label{rem:compact}
    Let $M$ be a diffuse finite von Neumann algebra and $N\subset M$ a diffuse von Neumann subalgebra.
    If $N\subset M$ is mixing, then it is easy to check that $e_N x JyJ e_N\in \B(L^2M)$ is a compact operator from $M$ to $L^2M$ assuming $x$ or $y\in M\ominus N$.
\end{rem}

Examples of mixing subalgebras include $M_1$ and $M_2\subset M_1\ast M_2$, where $M_1$ and $M_2$ are diffuse \cite[Proposition 1.6]{jolissaint}  and $L\Lambda\subset L\Gamma$, where $\Lambda<\Gamma$ is almost malnormal (see Proposition 2.4 in \cite{boutonnetcarderimaxamen}).  

The following \cite[Corollary 2.12]{Ioa15} is crucial to the proof of Theorem \ref{thm:subalg in free product}.  

\begin{lem}[Ioana]\label{Ioana}
Let $M_1$, $M_2$ be two diffuse tracial von Neumann algebras and $M=M_1*M_2$ be the tracial free product. 
Let  $A\subset M$ be a subalgebra such that $A$ is amenable relative to $M_1$ in $M$. Then either $A\prec_M M_1$  or $A$ is amenable.
\end{lem}

We also need the following case of the main result of \cite{BH18}: 

\begin{thm}[Boutonnet-Houdayer]\label{boutonnet houdayer theorem}

Let $M=M_1*M_2$, where $M_i$ are diffuse tracial von Neumann algebras. If $A\subset M$ is a von Neumann subalgebra that satisfies $A\cap M_1$ is diffuse and $A$ is amenable relative to $M_1$ inside $M$, then $A\subset M_1$.   
\end{thm}

\section{Proper proximality for von Neumann algebras and  boundary pieces}

\subsection{Boundary pieces from von Neumann subalgebras}

Let $M$ be a finite von Neumann algebra. 
An $M$-boundary piece is a hereditary ${\rm C}^*$-subalgebra $\X \subset \B(L^2M)$ such that $\M(\X) \cap M$ and $\M(\X) \cap JMJ$ are weakly dense in $M$ and $JMJ$, respectively, where $\M(\mathbb X)$ is the multiplier algebra of $\mathbb X$. 
To avoid pathological examples, we will always assume that $\X \not= \{ 0 \}$, 
and it follows that $\K(L^2M) \subset \X$, by the assumption on $\M(\X)$.

The main example of an $M$-boundary piece we use in this paper is one generated by von Neumann subalgebras. 
We recall some facts about hereditary ${\rm C}^*$-algebras for what follows (see e.g. \cite[II.5]{Bl06}).

Let $A$ be a ${\rm C}^*$-algebra. 
There is a one-to-one correspondence between the set of hereditary ${\rm C}^*$-subalgebras of A 
	and the set of closed left ideals in $A$: 
	given a hereditary ${\rm C}^*$-subalgebra $H\subset A$ , $L_H :=AH=\{ah\ |\ a\in A,h\in H\}$ is a closed left ideal; 
	and for a closed left ideal $L \subset A$, $H_L = L \cap L^*$ is a hereditary ${\rm C}^*$-subalgebra of $A$. 
Given a subset of operators $\{b_i\}_{i\in I}\subset A$,
	the hereditary ${\rm C}^*$-subalgebra generated by $\{b_i\}_{i\in I}$ is 
	$BAB = \{bab\ | \ b\in B_+, a\in A\}$, where $B$ is the ${\rm C}^*$-subalgebra generated by $\{b_i\}_{i\in I}$.

 \begin{examp}\label{examp:bdrysub}[Boundary piece generated by subalgebras]
Let $M$ be a finite von Neumann algebra.
Suppose $M_i \subset M$, $i=1,\hdots, n$ are von Neumann subalgebras 
	and denote by $e_{M_i} \in \B(L^2M)$ the orthogonal projection from $L^2M$ onto the space $L^2M_i$. 
The $M$-boundary piece associated with the family of subalgebras $\{M_i\}_{i=1}^n$
is the hereditary ${\rm C}^*$-subalgebra of $\B(L^2M)$ generated by operators of the form $x JyJ e_{M_i}$ with $x, y \in M$,  $i=1,\dots,n$, 
	and it is clear that $M$ and $JMJ$ are contained in its multiplier algebra.
 \end{examp}
 
 \begin{lem}
 Let $M$ be a finite von Neumann algebra and $M_i\subset M$, $i=1,\dots, n$ von Neumann subalgebras such that the projections
    $\{e_{M_i}\}_{i=1}^n$ are pairwise commuting. 
Let $\X$ be the hereditary ${\rm C}^*$-subalgebra in $\bB(L^2M)$ generated by $\{xJyJ(\vee_{i=1}^{n} e_{M_i})   \ | \ x,y\in M\}$ 
and $\Y$ the hereditary ${\rm C}^*$-subalgebra in $\bB(L^2M)$ generated by $\{xJyJe_{M_i}   \ | \ i=1,\cdots n, \ x,y\in M\}$. Then $\X=\Y$.
 \end{lem}

\begin{proof}
First note that $e_{M_i}\in \X$ for each $i$ since $0\leq e_{M_i}\leq \vee_{i=1}^{n} e_{M_i}$. 
We also have $\vee_{i=1}^{n} e_{M_i}\in \Y$.
In fact, for each pair $i,j$, $e_{M_i}\wedge e_{M_j}\in \Y$ as $0\leq e_{M_i}\wedge e_{M_j}\leq e_{M_i}$, 
	and $e_{M_i}\vee e_{M_j=}e_{M_i}+e_{M_j}-e_{M_i}\wedge e_{M_j}\in \Y$
	as $[e_{M_i}, e_{M_j}]=0$.   
To see that $\X\subset \Y$, note that $L=\bB(L^2M)\X$ is contained in $K=\bB(L^2M)\Y$. 
Indeed, for any $x,y\in M$ and $T\in \bB(L^2M)$, we have $T(\vee_{i=1}^{n} e_{M_i})xJyJ\in \bB(L^2M)\Y xJyJ = \bB(L^2M)\Y$ as $M$ and $JMJ$ are in the multiplier algebra of $\Y$. 
By a similar argument we see that $\Y\subset \X$. 
\end{proof}

\begin{lem}\label{boundary piece equivalent defn approximate unit}
Under the above assumption,
	$\{\vee_{u,v\in F}uJvJ(\vee_{i=1}^ne_{M_i})Jv^*Ju^*\}_{F\in \mathcal{F}}$ is an approximate unit for $\X$,
	where $\mathcal{F}$ is the collection of finite subsets of $\mathcal{U}(M)$ ordered by inclusion.
\end{lem}

\begin{proof}
Set $e_0=\vee_{i=1}^ne_{M_i}$ and $e_F=\vee_{u,v\in F}uJvJ e_0 Jv^*Ju^*$ for any $F\in\mathcal F$.
First we observe that $e_F\in \X$ as $0\leq e_F\leq \sum_{u,v\in F}uJvJ e_0 Jv^*Ju^*$. 
Note that whenever $1\in F$, $e_0e_F=e_0$ and hence $(u_0Jv_0Je_0)e_F=u_0Jv_0Je_0$, for any $u_0,v_0\in \mathcal{U}(M)$. 
On the other hand, if $u_0,v_0\in F$, we have $e_F(u_0Jv_0Je_0)=e_F(u_0Jv_0Je_0Jv_0^*Ju_0^*)u_0Jv_0Je_0=u_0Jv_0Je_0$. The result follows by writing arbitrary $x,y\in M$ as sums of four unitaries.  
\end{proof}
 
Fix an $M$-boundary piece $\X$ and let $\K^L_\X(M)\subset \B(L^2M)$ denote 
	the $\|\cdot\|_{\infty,2}$ closure of the closed left ideal $\B(L^2M)\X$, i.e., 
	$\K^L_\X(M)=\overline{\B(L^2M)\X}^{\|\cdot \|_{\infty,2}}$,
	where $\|\cdot\|_{\infty,2}$ on $\B(L^2M)$ is given by $\|T\|_{\infty,2}=\sup_{x\in (M)_1}\|T\hat x\|_2$ for $T\in\B(L^2M)$.

We let $\K_\X(M)=(\K_\X^L(M))^*\cap (\K_\X^L(M))$, which is a hereditary C$^*$-subalgebra of $\B(L^2M)$ with $M$ and $JMJ$ contained in $\M(\K_\X^L(M))$  \cite[Section 3]{DKEP21}.
Denote by $\K_\X^{\infty,1}(M)$ the $\|\cdot\|_{\infty, 1}$ closure of $\K_\X(M)$ in $\B(L^2M)$, 
	$\|T\|_{\infty, 1}=\sup_{x,y\in (M)_1} \langle T\hat{x},\hat y\rangle$ for $T\in \B(L^2M)$
and it coincides with $\overline{\X}^{\|\cdot\|_{\infty,1}}$.

Now put $\bS_{\X}(M) \subset \B(L^2M)$ to be
\[
\bS_{\X}(M) = \{ T \in \B(L^2M) \mid [T, J x J] \in \K_\X^{\infty,1}(M) {\rm \ for \ all \ } x \in M \},
\]
which is an operator system that contains $M$. 
In the case when $\X = \K(L^2M)$, we write $\bS(M)$ instead of $\bS_{\K(L^2M)}(M)$. 

Recall from \cite[Theorem 6.2]{DKEP21} that for a finite von Neumann subalgebra $N\subset M$ and an $M$-boundary piece $\mathbb X$, 
	we say $N$ is properly proximal relative to $\X$ in $M$ if there is no $N$-central state $\varphi$ on $\bS_\X(M)$ that is normal on $M$.
And we say $M$ is properly proximal if $M$ is properly proximal relative to $\mathbb K(L^2M)$ in $M$.

\begin{rem}\label{remark}
Let $M$ and $Q$ be finite von Neumann algebras, $\X$ an $M$-boundary piece, 
and $N\subset pMp$ be a von Neumann subalgebra, where $0\neq p\in\mathcal P(M)$.
\begin{enumerate}
\item \label{rem:induced boundary piece} 
	Consider the u.c.p. map $\mathcal E_N :=\Ad(e_N)\circ \Ad(pJpJ):\B(L^2M)\to \B(L^2N)$.
 	Then by \cite[Remark 6.3]{DKEP21} that $\mathcal E_N(\K_\X(M))\subset \B(L^2N)$ forms an $N$-boundary piece.
	And we say $\mathcal E_N(\K_\X(M))$ is the induced $N$-boundary piece, which will be denoted by $\X^N$.
\item \label{rem:sum of prop prox} If $N$ is properly proximal relative to $\X$ inside $M$,
		then $zN$ is also properly proximal relative to $\X$ inside $M$ for any $0\neq z\in\mathcal Z(\mathcal P(N))$,
		since $\Ad(z)\circ \bS_\X(M)\subset \bS_\X(M)$.
\item \label{rem:rel prop prox implies no amen sum} If $N$ is properly proximal relative to $\X$ inside $M$, 
			then $N$ has no amenable direct summand.
		To see this, suppose $qN$ is amenable for some $0\neq q\in \cZ(\mathcal P(N)$ 
			and let $\varphi$ be a $qN$-central state on $\B(L^2(qN))$.
		Consider $\mu:=\varphi\circ \Ad(q) \circ \Ad(e_N):\B(L^2M)\to \C$,
			and one checks that $\mu$ is a $N$-central state with $\mu_{|M}$ being normal.
\item \label{rem:maximal prop prox sum} Notice that from the definition it follows that proper proximality is stable under taking direct sum.
		Thus we may take $f\in \cZ(\mathcal P(Q))$ so that $Qf$ is the maximal properly proximal direct summand of $Q$.
\end{enumerate}

\end{rem}

\subsection{Bidual formulation of proper proximality}\label{sec: bidual}
Given a finite von Neumann algebra $M$ and a C$^*$-subalgebra $A\subset \B(L^2M)$ such that $M$ and $JMJ$ are contained in $\M(A)$,
we recall that $A^{M \sharp M}$ (resp. $A^{JMJ\sharp JMJ})$ denotes the space of $\varphi \in A^*$ 
		such that for each $T \in A$ the map $M \times M \ni (a, b) \mapsto \varphi(aTb)$ 
		(resp. $JMJ\times JMJ \ni (a,b)\mapsto \varphi(aTb)$) is separately normal in each variable 
		and set $A^{\sharp}_J = A^{M \sharp M} \cap A^{JMJ \sharp JMJ}$.
Moreover, we may view $(A^\sharp_J)^*$ as a von Neumann algebra in the following way, as shown in \cite[Section 2]{DKEP21}.
Denote by $p_{\rm nor}\in \B(L^2M)^{**}$ the supremum of support projections of states in $\B(L^2M)^*$ 
	that restrict to normal states on $M$ and $JMJ$, 
	so that $M$ and $JMJ$ may be viewed as von Neumann subalgebras of $p_{\rm nor} \M(A)^{**} p_{\rm nor}$.
  Note that $p_{\rm nor}$ lies in $\M(A)^{**}$ and 
    $p_{\rm nor} \M(A)^{**} p_{\rm nor}$ is canonically identified with $(\M(A)^\sharp_J)^*$. 
Let $q_A\in\mathcal P(\M(A)^{**})$ be the central projection such that $q_A(\M(A)^{**})=A^{**}$ 
	and we may then identify $(A^\sharp_J)^*$ with $q_A p_{\rm nor} \M(A)^{**} p_{\rm nor}=p_{\rm nor} A^{**} p_{\rm nor}$,
    which is also a von Neumann algebra.
Furthermore, if $B\subset A$ is another C$^*$-subalgebra with $M$, $JMJ\subset \M(B)$, 
	we may identify $(B^\sharp_J)^*$ with $q_B p_{\rm nor} A^{**} p_{\rm nor} q_B$,
	which is a non-unital subalgebra of $(A^\sharp_J)^*$.

We will need the following bidual characterization of properly proximal. 

\begin{lem}{\cite[Lemma 8.5]{DKEP21}}\label{lem:bidual character}
Let $M$ be a separable tracial von Neumann algebra with an $M$-boundary piece $\mathbb X$.
Then $M$ is properly proximal relative to $\X$ if and only if there is no $M$-central state $\varphi$ on
\[
\widetilde{\bS}_{\X}(M) := \left\{ T \in \left(\B(L^2 M)_J^{{\sharp}} \right)^* \mid [T, a] \in \left( \K_\X(M)_J^\sharp \right)^* \ {\rm for \ all \ } a \in JMJ \right\}
\] 
such that $\varphi_{| M }$ is normal.
\end{lem}

Using the above notations, we observe that we may identify $\tilde{\bS}_{\X}(M)$ in the following way:
\[\begin{aligned}
    \widetilde{\bS}_\X(M) &= \{ T\in \big(\bB(L^2 M )_{J}^\sharp\big)^* | \ [T,a]\in \big(\K_\X( M )_{J}^\sharp\big)^*, {\rm\ for\ any}\ a\in JMJ\}\\
    &= \{ T\in p_{\rm nor}\bB(L^2M)^{**}p_{\rm nor}\ | \ [T,a]\in q_{\X}p_{\rm nor}\big(M(\K_\X(M))\big)^{**}p_{\rm nor}q_\X , {\rm\ for\ any}\ a\in JMJ\},
\end{aligned}\]
where $q_{\X}$ is the identity of $\K_\X(M)^{**}\subset \big(M(\K_\X(M))\big)^{**}$.
If we set $q_\K=q_{\K(L^2M)}$ to be the identity of $\K(L^2M)^{**}\subset \B(L^2M)^{**}$,
	then using the above description of $\tilde \bS_\X(M)$, 
	we have  $q_{\X}^{\perp} \widetilde{\bS}_{\X}(M)q_{\X}^{\perp}\subset q_{\K}^\perp \widetilde{\bS}(M)$,
	as $q_\X$ commutes with $JMJ$.

\begin{rem}\label{rem:bidual notation}
Recall that we may embed $\B(L^2M)$ into $(\B(L^2M)^\sharp_J)^*$ through the u.c.p.\ map $\iota_{\rm nor}$,
	which is given by $\iota_{\rm nor}=\Ad(p_{\rm nor})\circ \iota$,
	where $\iota:\B(L^2M)\to \B(L^2M)^{**}$ is the canonical $*$-homomorphism into the universal envelope,
	and $p_{\rm nor}$ is the projection in $\B(L^2M)^{**}$ such that $p_{\rm nor} \B(L^2M)^{**} p_{\rm nor} = (\B(L^2M)^\sharp_J)^*$.
Restricting $\iota_{\rm nor}$ to ${\rm C}^*$-subalgebra $A\subset \B(L^2M)$ satisfying $M, JMJ\subset \M(A)$
	give rise to the embedding of $A$ into $(A^{\sharp}_J)^*$,
	and $(\iota_{\rm nor})_{\mid M}$, $(\iota_{\rm nor})_{\mid JMJ}$ are faithful normal representations of $M$ and $JMJ$, respectively.
Furthermore, although $\iota_{\rm nor}$ is not a $*$-homomorphism,  
	${\rm sp}M e_B M$ (that is, the span of elements $xe_By$ where $x,y\in M$) is in the multiplicative domain of $\phi_0$.
\end{rem}
\begin{lem}\label{lem: approx unit}
Let $M$ be a finite von Neumann algebra and $\X$ an $M$-boundary piece.
Let $\X_0\subset \K_\X(M)$ be a ${\rm C}^*$-subalgebra and $\{e_n\}_{n\in I}$ an approximate unit of $\X_0$.
If $\X_0\subset \K_\X^{\infty,1}(M)$ is dense in $\|\cdot\|_{\infty,1}$ 
	and $\iota(e_n)$ commutes with $p_{\rm nor}$ for each $n\in I$,
	then $\lim_n \iota_{\rm nor}( e_n)\in (\K_\X(M)^\sharp_J)^*$ is the identity, where the limit is in the weak$^*$ topology.
\end{lem}
\begin{proof}
Since $\iota_{\rm nor}(\K_\X(M))\subset (\K_\X(M)^\sharp_J)^*$ is weak$^*$ dense and functionals in $\K_\X(M)^\sharp_J$ are continuous in $\|\cdot\|_{\infty,1}$ topology by \cite[Proposition 3.1]{DKEP21}, 
	we have $\iota_{\rm nor}(\X_0)\subset (\K_\X(M)^\sharp_J)^*$ is also weak$^*$ dense.
Let $e=\lim_n \iota_{\rm nor}(e_n)\in (\K_\X(M)^\sharp_J)^*$ be a weak$^*$ limit point and for any $T\in \X_0$, we have
$$e\iota_{\rm nor}(T)=\lim_n p_{\rm nor} \iota(e_n)\iota(T) p_{\rm nor} =\lim_n p_{\rm nor} \iota (e_nT) p_{\rm nor}=\iota_{\rm nor}(T),$$
and similarly $\iota_{\rm nor}(T) e= \iota_{\rm nor}(T)$.
By density of $\iota_{\rm nor}(\X_0)\subset (\K_\X(M)^\sharp_J)^*$, we conclude that $e$ is the identity in $(\K_\X(M)^\sharp_J)^*$.
\end{proof}

\begin{lem}\label{lem: projection commute}
Let $M$ be a finite von Neumann algebra and $N\subset M$ a von Neumann subalgebra.
Let $e_N\in \B(L^2M)$ be the orthogonal projection onto $L^2N$.
Then $\iota (e_N) \in \B(L^2M)^{**}$ commutes with $p_{\rm nor}$.
\end{lem}
\begin{proof}
Suppose $\B(L^2M)^{**}\subset \B(\cH)$ and notice that $\xi \cH$ is in the range of $p_{\rm nor}$ 
	if and only if $M\ni x\to \langle \iota(x)\xi ,\xi\rangle$ and $JMJ\ni x\to \langle \iota(x)\xi ,\xi\rangle$
	are normal.
For $\xi\in p_{\rm nor} \cH$, we have $\varphi(x):=\langle \iota(x) \iota(e_B)\xi, \iota(e_N) \xi\rangle=\langle \iota(E_N(x))\xi, \xi\rangle$
	is also normal for $x\in M$ and $JMJ$,
	which implies that $\iota(e_N) p_{\rm nor}=p_{\rm nor} \iota(e_N) p_{\rm nor}$.
It follows that $\iota(e_N)$ and $p_{\rm nor}$ commutes.
\end{proof}

\begin{lem}\label{lem: identity}
Let $N\subset M$ be a mixing von Neumann subalgebra admitting a Pimsner-Popa basis $\{m_k\}$ where $m_k\in M$. Let $\X_N$ be the associated boundary piece (see Example \ref{examp:bdrysub}), and $q_{\K}\in (\K(L^2M)_J^{\sharp})^*$, $q_{\X_N}\in (\K_{\X_N}(M)_J^{\sharp})^*$ be the respective identity elements. Then $$\sum_{k,l}q_{\K}^{\perp}\iota_{\rm nor}(m_kJm_l^*Je_NJm_lJm_k^*)=q_{\K}^{\perp}q_{\X_{N}}.$$
\end{lem}

\begin{proof}
For notational simplicity, denote by $p_{k,l}= q_{\K}^{\perp}\iota_{\rm nor}(m_kJm_l^*Je_NJm_lJm_k^*)$. By mixing property of the inclusion $N\subset M$, we see that $p_{k,l}$ are pairwise orthogonal projections. Indeed, if $N\subset M$ is mixing, we have $e_NxJyJe_N-e_NE_N(x)JE_N(y)Je_N\in \K(M)$, i.e, is a compact operator when viewed as a bounded operator from the normed space $M$ to $L^2(M)$. 
Now we compute \[\begin{aligned}
    p_{k,l}p_{k',l'}&= q_{\K}^{\perp}\iota_{\rm nor} (m_k Jm_lJe_NJm_l^*m_{l'}Jm_k^*m_{k'}e_NJm_{l'}^*Jm_{k'}^*)\\
    &= q_{\K}^{\perp}\iota_{\rm nor} (m_k Jm_lJe_NJq_lJq_ke_NJm_{l'}^*Jm_{k'}^*)\delta_{k,k'}\delta_{l,l'}\\ &= q_{\K}^{\perp}\iota_{\rm nor}(m_kJm_l^*Je_NJm_{l'}^*Jm_{k'})\delta_{k,k'}\delta_{l,l'}
\end{aligned}\]

where $q_l\in \mathcal{P}(N)$ such that $q_l=E_N(m_l^*m_l)$ and automatically satisfies   $m_lq_l= m_l$ (see Section \ref{PP}).

Denote by $\X_0\subset \B(L^2 M)$ the hereditary ${\rm C}^*$-subalgebra generated by $ x JyJ e_N$ for $x,y$ in the ${\rm C}^*$-algebra  $A$    generated by   $\{m_ka\}_{a\in  N,  k\in \mathbb{N}}$.
It is clear that $\X_0$ is an $M$-boundary piece 
and note that $A$  is weakly dense (see Section 2.1, (\ref{PPbasisfact2})) in $M$.  

Observe that $\K_{\X_0}^{\infty,1}(M)=\K_{\X_{N}}^{\infty,1}(M)$
, where $\K_{\X_0}^{\infty,1}(M)$ is obtained from $\X_0$.
Notice that $\B(L^2 M) \X_0 \subset \K_\X^L(M)$ is dense in $\|\cdot\|_{\infty,2}$.
Indeed, for any contractions $T\in \B(L^2 M)$ and $x,y\in M$, 
	we may find a net of contractions $T_i\in \B(L^2 M)\X_0$ such that $T_i\to Te_N x JyJ$
	in $\|\cdot\|_{\infty,2}$,
	as it follows directly from \cite[Proposition 3.1]{DKEP21}.
It then follows that $\K_{\X_0}(M)\subset \K_\X^{\infty,1}(M)$ is dense in $\|\cdot\|_{\infty,1}$
	and hence $\overline{\X_0}^{\infty,1}=\K_{\X_0}^{\infty,1}(M)=\K_\X^{\infty,1}(M)$ by \cite[Proposition 3.6]{DKEP21}. Note that $p_{k,l}\in (\K_{\X_0}(M)_J^{\sharp})^*= (\K_{\X_N}(M)_J^{\sharp})^*$ and $p_{k,l}\leq q_{\K}^{\perp}q_{\X_N}$.

 By the  above paragraph  it suffices to check the following: $(\sum_{k',l'} p_{k',l'} )\iota_{\rm nor}(m_k J m_\ell J aJb Je_N)= q_{\K}^{\perp}\iota_{\rm nor}(m_k J m_\ell J   aJbJe_N)$ and $\iota_{\rm nor}(e_N   Jm_\ell J m_kaJb J) (\sum_{k',l'} p_{k',l'} )=q_{\K}^{\perp}\iota_{\rm nor}(e_N Jm_\ell J m_kaJb J)$ for all $a,b\in N$ and $k,l \in \mathbb{N}$. Indeed, every element in $\X_0$ can be written as a norm limit of linear spans consisting of elements of the from $x_1 Jy_1 J T Jy_2 J x_2$,
	where $x_i,y_i\in A$ and $T\in \B(L^2N)$
    Further we can assume $x_i=m_ka$ with $a\in N$ from density.  Then we will get that for all $z\in \X_0$, $\sum_{k,l}q_{\K}^{\perp}\iota_{\rm nor}(m_kJm_l^*Je_NJm_lJm_k^*)\iota_{\rm nor}(z)= q_{\K}^{\perp}\iota_{\rm nor}(z)$ and  since $\iota_{\rm nor}(\X_0)$ is weak$^*$ dense in $(\K_{\X_0}(M)_J^{\sharp})^*= (\K_{\X_N}(M)_J^{\sharp})^*$ by the previous paragraph, so we get that $\sum_{k,l}q_{\K}^{\perp}\iota_{\rm nor}(m_kJm_l^*Je_NJm_lJm_k^*)= q_{\K}^{\perp}q_{\X_{N}}$. 

 The above equality holds by a simple computation  
$$p_{k',l'} \iota_{\rm nor}(m_k J m_\ell J aJb Je_N)=
\delta_{k,k'}\delta_{l,l'}q_{\K}^{\perp}\iota_{\rm nor}(m_k J m_\ell J   aJbJe_N)$$ as in the beginning of this proof wherein we verified that $p_{k,l} $ are projections.\end{proof}

\begin{lem}\label{sup of projections}
    Let $M$ be a finite von Neumann algebra and $M_i\subset M$, $i=1, \dots n$ be von Neumann subalgebras such that $e_{M_i}$ are pairwise commuting. Let $\X$ denote the boundary piece associated to $\{M_i\}_{i=1}^{n}$ as in Example \ref{examp:bdrysub}. Let $\X_i$ denote the boundary pieces associated to $M_i$.  Let $q_i$ denote the identities of the von Neumann algebras   $( \K_{\X_i}(M)_J^\sharp )^*$ and  $q_\X$  denote the identity of   $( \K_\X(M)_J^\sharp )^*$.   Then we have that $q_\X= \vee_{i=1}^nq_i$.
\end{lem}

\begin{proof}
Recall from the beginning of this section that $( \K_\X(M)_J^\sharp )^*$ is a von Neumann algebra, as $M, JMJ$ are in the multiplier algebra of  $\M( \K_\X(M))$.
It is easy to see that $q_\X\geq q_i$ for each $i$. Now we show that $q_\X\leq \vee_{i=1}^nq_i$. Fix an increasing family of finite subsets of unitaries in $M$, $F_n$ such that $\{\bigcup_{n}F_n\}''=M$. 
Let $e_{n}= \vee_{u,v\in F_n} u Jv J (\vee_{i} e_{M_i} )Jv ^*J u^*$. Clearly we have that $\iota_{\rm nor}(e_n)\leq \vee_{i=1}^nq_i$. 
Indeed, see that 
	$$\iota_{\rm nor}(\vee_{u,v\in F_n} u Jv J e_{M_i}Jv ^*J u^*)<q_i$$ 
	and then $\iota_{\rm nor}(e_{n})=\iota_{\rm nor}(\vee_{i}\vee_{u,v\in F_n} u Jv J e_{M_i}Jv ^*J u^*)\leq \vee_{i=1}^nq_i$. 
From Lemmas \ref{boundary piece equivalent defn approximate unit}  and \ref{lem: approx unit}  we see that $q_{\X}= \lim_{n} \iota_{\rm nor}(e_{n})\leq \vee_{i=1}^nq_i$    as required.  
\end{proof}


\subsection{Induced boundary pieces in the bidual}\label{induced bdry pieces}

\begin{lem}\label{lem:cond exp}
Let $M$ be a finite von Neumann algebra,  $\X$ an $M$-boundary piece,
	and $N\subset pMp$ a von Neumann subalgebra for some $0\neq p\in \PP(M)$.
Set $E:=\Ad(e_N)\circ \Ad(pJpJ):\B(L^2M)\to \B(L^2N)$.
Then its restriction $E_{| \bS(M)}$ maps $\bS_\X(M)$ to $\bS(N)$.
Moreover, there exists a u.c.p.\ map $\tilde E: \tilde \bS(M)\to \tilde \bS(N)$ such that $\tilde E_{\mid M}$ agrees with the conditional expectation from $M$ to $N$.
\end{lem}
\begin{proof}
To see $E_{| \bS_\X(M)}: \bS(M) \to \bS(N)$, note that $p JpJ e_N J_N a J_N= J a J p JpJ e_N$ for any $a\in N$.
	and $E: \B(L^2M)\to \B(L^2N)$ is $\|\cdot\|_{\infty, 1}$-continuous.
Thus for any $T\in \bS_\X(M)$ and any $a\in N$, we have
$$ [E(T), J_N aJ_N]=E([T, JaJ])\in E(\overline {\K(M)}^{\|\cdot\|_{\infty,1}})=\overline{\K(N)}^{\|\cdot\|_{\infty,1}}=\K^{\infty,1}(N),$$
i.e., $E(T)\in \bS_(N)$.

Note that $E^*: \B(L^2N)^*\to \B(L^2M)^*$ maps $\B(L^2N)^{\sharp}_J$ to $\B(L^2M)^{\sharp}_J$ by \cite[Lemma 5.3]{DKEP21},
	and similarly $E^*: (\K(L^2N))^\sharp_J\to (\K(L^2M))^\sharp_J$.
Therefore $\tilde E:=({E^*}_{|\B(L^2N)^{\sharp}_J})^*: (\B(L^2M)^{\sharp}_J)^*\to (\B(L^2N)^{\sharp}_J)^*$
	and $\tilde E_{\mid (\K(L^2M))^\sharp_J)^*}: (\K(L^2M)^\sharp_J)^*\to (\K(L^2N)^\sharp_J)^*$.
Hence we conclude that $\tilde E: \tilde \bS(M)\to \tilde \bS(N)$  
	with $\tilde E_{\mid M}$ agrees with the conditional expectation from $M$ to $N$.
\end{proof}


\subsection{Relative biexactness and relative proper proximality}

Given a countable discrete group $\Gamma$, a boundary piece $I$ is a $\Gamma\times \Gamma$ invariant closed ideal such that $c_0\Gamma\subset I\subset \ell^\infty\Gamma$ \cite{BIP18}.  
The small at infinity compactification of $\Gamma$ relative to $I$ is the spectrum of the ${\rm C}^*$-algebra $\bS_I(\Gamma)=\{f\in\ell^\infty\Gamma\mid f-R_tf\in I, {\rm\ for\ any\ } t\in\Gamma\}$.
Recall that $\Gamma$ is said to be biexact relative to $X$ if $\Gamma \actson \bS_I(\Gamma)/I$ is topologically amenable 
	\cite{Oza04}, \cite[Chapter 15]{BO08}, \cite{BIP18}.
We remark that this is equivalent to $\Gamma\actson\bS_I(\Gamma)$ is amenable.
Indeed, since we may embed $\ell^\infty\Gamma \hookrightarrow I^{**}$ in a $\Gamma$-equivariant way, 
	we have $\Gamma\actson I^{**} \oplus (\bS_I(\Gamma)/I)^{**} = \bS_I(\Gamma)^{**}$ is amenable,
	and it follows that $\Gamma\actson \bS_I(\Gamma)$ is an amenable action \cite[Proposition 2.7]{BEW19}.

The following is a general version of \cite[Theorem 7.1]{DKEP21}, whose proof follows similarly.   For the convenience of the reader we include the proof sketch below.
A more general version of this is obtained in the upcoming work    \cite{DP22}.

\begin{thm}\label{thm:biexact dichotomy}
Let $M=L\Gamma$ where $\Gamma$ is an nonamenable group that is biexact relative to a finite family of subgroups $\{\Lambda_i\}_{i\in I}$. 
Denote by $\X$ the $M$-boundary piece associated with $\{L\Lambda_i\}_{i\in I}$.
If $A\subset pMp$ for some $0\neq p\in \mathcal{P}(M)$ such that $A$ has no amenable direct summands, 
then $A$ is properly proximal relative to $\X^A$, where $\X^A$ is the induced $A$-boundary piece as in 
Remark~\ref{remark}).
\end{thm}

\begin{proof}
    Consider the $\Gamma$-equivariant diagonal embedding $\ell^\infty(\Gamma)\subset \B(\ell^2\Gamma)$. Note that under this embedding   $c_0(\Gamma,\{\Lambda_i\}_{i\in I})$ is sent to $\X$. Denote by $\bS_{\X}(\Gamma)=\{f\in \ell^\infty(\Gamma)|\ f-fg\in c_0(\Gamma,\{\Lambda_i\}_{i\in I}),\ \forall g\in \Gamma\}$, the relative small at infinity compactification  at the group level.   Restricting this embedding to $\bS_{\X}(\Gamma)$ then gives a $\Gamma$-equivariant embedding into $\bS_\X(M)$. Therefore we obtain a $*$-homomorphism from $\bS_{\X}(\Gamma)\rtimes_r \Gamma\to \B(\ell^2(\Gamma)) $ whose image is contained in $ \bS_{\X}(M)$. Composing this with the map $E$ from Lemma \ref{lem:cond exp}, we obtain a u.c.p map $\phi: \bS_{\X}(\Gamma)\rtimes_r \Gamma\to \bS_{\X^{A}}(A)$.     By hypothesis we have a projection $p_0\in \mathcal{Z}(A)$ and an $Ap_0$ bimodular u.c.p map $\Phi: \bS_{\X^A}(A)\to Ap_0 $. Further composing with this map we obtain a u.c.p map from $\widetilde{\phi}:\bS_{\X}(\Gamma)\rtimes_r \Gamma\to Ap_0$.  

    Now set $\varphi: \bS_{\X}(\Gamma)\rtimes_r \Gamma \to \C$, by $\varphi(x):=\frac{\langle x\widehat{p_0},\widehat{p_0} \rangle }{\tau(p)}$. We then get a representation $\pi_\varphi:\bS_{\X}(\Gamma)\rtimes_r \Gamma \to \mathcal{H}_{\varphi} $ and a state $\widetilde{\varphi}\in \B(\mathcal{H}_{\varphi})_{*}$ such that $\varphi=\widetilde{\varphi}\circ \pi_{\varphi}$. Since $C^*_r(\Gamma)$ is weakly dense in $M$, we see by an argument of Boutonnet-Carderi (see Propositon 4.1 in \cite{BC15}) that there is a projection $q\in (\pi_\varphi(\bS_{\X}(\Gamma)\rtimes_r \Gamma))''$ such that  $\widetilde{\varphi}(q)=1$ and there exists a normal unital $*$-homomorphism $\iota: L(\Gamma)\to q\pi_\varphi(\bS_{\X}(\Gamma)\rtimes_r \Gamma))''q$.       
    
    Since $\Gamma$ is biexact relative to $\X$, we have that $\bS_{\X}(\Gamma)\rtimes_r \Gamma$ is a nuclear $C^*$-algebra. Therefore there is a u.c.p map $\widetilde{\iota}:\B(\ell^2(\Gamma))\to  q(\bS_{\X}(\Gamma)\rtimes_r \Gamma)''q$    extending $\iota$. Now we see that $\widetilde{\varphi}\circ\widetilde{\iota}$ is an $Ap_0$ central state on $\B(\ell^2(\Gamma))$ showing that $A$ has an amenable direct summand, which is a contradiction.
\end{proof}



In the case of general free products of finite von Neumann algebras $M=M_1*M_2$ it ought to be the case that that if $A\subset M$ such that $A$ has no amenable direct summand, then $A\subset M$ is properly proximal relative to the boundary piece generated by $M_1$ and $M_2$. However, currently we are only able to obtain this with an additional technical assumption that $M_i\cong L \Gamma_i$ where $\Gamma_i$ are exact, so that $\Gamma_1\ast \Gamma_2$ is  biexact relative to $\{\Gamma_1, \Gamma_2\}$ \cite[Proposition 15.3.12]{BO08}. We record below a general result about subalgebras in free products which follows essentially from Theorem 9.1 in \cite{DKEP21}, however we do not get the boundary piece associated to the subalgebras $M_i$. We instead get the boundary piece associated to the word length:    

Let $M_1$, $M_2$ be two finite von Neumann algebras and $M=M_1* M_2$ be the tracial free product. 
Let  $A\subset M$ be a nonamenable subalgebra. Consider the free product deformation from \cite{IPP08}, i.e., $\tilde M =M\ast L\mathbb F_2$, $\theta_t=\Ad(u_1^t)\ast \Ad(u_2^t)\in {\rm Aut}(\tilde M)$, with $u_1^t=\exp( i t\alpha_1)$, $u_2^t= \exp( i t\alpha_2)$, where $\alpha_1$, $\alpha_2$ are selfadjoint element in $L\mathbb F_2$ such that $\exp(i \alpha_1)=u_1$, $\exp(i \alpha_2)=u_2$ and $u_1$, $u_2$ are Haar unitaries in $L\mathbb F_2$. For $t>0$, we have $E_M\circ \alpha_t=P_0+ \sum_{n=1}^\infty (\sin(\pi t)/\pi t)^{2n} P_{n}$ (see Section 2.5 in \cite{Ioa15}), where $P_n$ is the orthogonal projection to $\mathcal H_n=\oplus_{(i_1,\cdots, i_n)\in S_n} L^2(M_{i_1}\ominus \mathbb C)\otimes \cdots\otimes L^2(M_{i_n}\ominus \mathbb C)$ and $S_n$ is the set of alternating sequences of length $n$. Consider the hereditary $C^*$-algebra $\mathbb X_F$ generated by $\{P_n\}_{n\geq 0}$.  

\begin{prop}

 In the above setup, there exists a projection $p\in A$ such that $Ap$ is amenable and  $Ap^{\perp}$ is properly proximal relative to $\X_F$.   

\end{prop}

\begin{proof}
It follows from the proof of \cite[Proposition 9.1]{DKEP21} that there exists an $M$-bimodular u.c.p. map $\phi: (M^{\rm op})'\cap \mathbb B(L^2\tilde M\ominus L^2 M)\to \tilde {\mathbb S}_{\mathbb X_F}(M)$.
Moreover, since $L^2\tilde M\ominus L^2M\cong L^2M\overline\otimes\mathcal K$ as an $M$-$M$ bimodule for some right $M$ module $\mathcal K$ \cite[Lemma 2.10]{Ioa15}, we may restrict $\phi$ to $\mathbb B(L^2M)\otimes \id_\mathcal K$.
Take $p\in\mathcal Z(A)$ to be the maximal projection such that $Ap$ is amenable and $p\neq 1$ as $A$ is nonamenable. If $Ap^{\perp}$ is not properly proximal relative to $\mathbb X$ inside $M$, i.e., there exists an $A$-central state $\varphi$ on $\tilde{\mathbb S}_{\mathbb X_F}(M)$ which is normal when restricted to $p^\perp Mp^\perp$.
Then pick $q\in \mathcal Z(Ap^\perp)$ be the support projection of $(\varphi\circ\phi)_{\mid Ap^\perp}$ and we have $Aq$ is amenable, which contradicts the maximality of $p$.

\end{proof}


\section{The Upgrading Theorem }

\begin{proof}[Proof of Theorem \ref{prop:bootstrap}]
    First notice that since $A$ is properly proximal relative to $\X$ inside $M$, 
	it has no amenable direct summand by (\ref{rem:rel prop prox implies no amen sum}) of Remark~\ref{remark}.
Let $f\in \mathcal Z(A)$ be the projection such that $Af^\perp$ is the maximal properly proximal direct summand of $A$ by (\ref{rem:maximal prop prox sum}) of  Remark~\ref{remark},
    and we may assume $f\neq 0$ since otherwise $A$ would be properly proximal.
Therefore $Af$ has no amenable direct summand, is properly proximal 
    relative to $\X$ inside $M$ by (\ref{rem:sum of prop prox}) of Remark~\ref{remark} 
	and has no properly proximal direct summand.
It follows from Lemma~\ref{lem:bidual character} that 
	there exists an $Af$-central state $\mu$ on $\widetilde{\bS}(Af)$ such that $\mu_{\mid Af}$ is normal.
Moreover, by a maximal argument, we may assume $\mu_{\mid \mathcal Z(Af)}$ is faithful, as $Af$ has no properly proximal direct summand.

Let $\tilde E:\tilde \bS(M)\to \tilde \bS(Af)$ be the u.c.p.\ map as in Lemma~\ref{lem:cond exp}.
Define a state $\varphi=\mu\circ \tilde E: \tilde \bS(M)\to \C$,
	and it follows that $\varphi$ is $Af$-central and $\varphi_{\mid fMf}$ is a faithful normal state.
Let $q_\K$ be the identity of the von Neumann algebra $(\K(L^2M)^\sharp_J)^*\subset (\B(L^2M)^\sharp_J)^*$, 
	$q_\X$ the identity of von Neumann algebra $(\K_\X(M)^\sharp_J)^* \subset (\B(L^2M)^\sharp_J)^*$.
Note that $q_\K\leq q_\X$ as $\K(L^2M)\subset \K_\X(M)$.

First we analyze the support of $\varphi$.
Observe that $\varphi(q_{\K}^{\perp})=1$.
Indeed, if $\varphi(q_{\K})> 0$, i.e., $\varphi$ does not vanish on $(\mathbb K (L^2M)^\sharp_J)^*$, 
	then we may restrict $\varphi$ to $\mathbb B(L^2M)$, 
	which embeds into $(\mathbb K (L^2M)^\sharp_J)^*$ as a normal operator $M$-system \cite[Section 8]{DKEP21}, 
	and this shows that $Af$ would have an amenable direct summand.
Moreover, we have $\varphi(q_{\X})=1$.
Indeed, if $\varphi(q_{\X}^\perp)>0$, 
	then 
$$\frac{1}{\mu(q_\X^\perp)}\varphi \circ \Ad(q_\X^\perp): \tilde {\mathbb S}_{\mathbb X}(M)\to\mathbb C$$ 
	would be an $Af$-central that restricts to a normal state on $fMf$.
Since $\bS_\X(M)$ naturally embeds into $\widetilde{{\mathbb S}}_{\mathbb X}(M)$, this contradicts that $Af$ is properly proximal relative to $\X$ inside $M$.
Therefore we conclude that $\varphi(q_\X q_\K^\perp)=1.$

For each $1\leq i\leq n$, denote by $\X_i:=\X_{M_i}\subset \B(L^2M)$ the $M$-boundary piece associated with $M_i$
	and $q_i\in (\K_{\X_i}(M)^\sharp_J)^*$ the identity.
Since $\vee_{i=1}^n q_i= q_\X$ by Lemma \ref{sup of projections},
    we have $\varphi(q_jq_{\K}^\perp)>0$ for some $1\leq j \leq n$

\noindent{\bf Claim:} there exists a u.c.p.\ map $\phi: \langle M, e_{M_j}\rangle \to q_\K^\perp q_j \tilde \bS(M) q_j$ 
	such that $\phi(x)=q_\K^\perp q_j x$ for any $x\in M$.

\begin{proof}[Proof of the claim.]
\renewcommand\qedsymbol{$\blacksquare$}
Denote by $\{m_k\}_{k\geq 0}\subset M$ a bounded Pimsner-Popa basis of $M$ over $M_i$.
For each $n\geq 0$, consider the u.c.p.\ map $\psi_n: \langle M, e_{M_j}\rangle \to \langle M, e_{M_j}\rangle$ given by 
	$$\psi_n(x)=(\sum_{k\leq n} m_k e_{M_j} m_k^*) x (\sum_{\ell \leq n} m_\ell e_{M_j} m_\ell^*),$$
	and notice that $\psi_n$ maps $\langle M, e_{M_j}\rangle$ into the $*$-subalgebra $A_0 :={\rm sp}\{ m_k a e_{M_j} m_\ell ^*\mid a\in M_j, k,\ell\geq 0\}$.

Recall notations from Remark \ref{rem:bidual notation}.
By Lemma~\ref{lem: projection commute}, we have 
$$\{\iota_{\rm nor}(Jm_k J e_{M_j} J m_k^* J)\}_{k\geq 0} \subset (\B(L^2M)^\sharp_J)^*$$
	is a family of pairwise orthogonal projections.
Set 
    $$e_j=\sum_{k\geq 0} \iota_{\rm nor}(Jm_k J e_{M_j} J m_k^* J) \in (\B(L^2M)^\sharp_J)^*$$ 
    and 
    define the map
\[\begin{aligned}
        \phi_0: A_0  &\to  q_{\K}^\perp (\B(L^2M)^\sharp_J)^*\\
        m_r a e_{M_j} m_\ell^*&\mapsto q_\K^\perp \iota_{\rm nor}(m_r a) e_{j} \iota_{\rm nor}(m_\ell^*).
\end{aligned}\]

It is easy to check that $\phi_0$ is well-defined.
We then check that $\phi_0$ is a $*$-homomorphism.
It suffices to show that for any $x\in M$, we have 
\begin{equation}\label{equa1}
q_\K^\perp e_{j} \iota_{\rm nor}(x) e_{j} = q_\K^\perp  \iota_{\rm nor}(E_{M_j}(x)) e_{j}.
\end{equation}
Now we compute, 
\[
\begin{aligned}
&q_\K^\perp  e_{j} \iota_{\rm nor}(x) e_{j} \\
=&q_\K^\perp  \sum_{k,\ell \geq 0} \iota _{\rm nor}\big ((Jm_k J e_{M_j} Jm_k^* J) x( Jm_\ell J e_{M_j} Jm_\ell ^*J)\big ) \\
=&q_\K^\perp \sum_{k\geq 0} \iota _{\rm nor}\big ((Jm_k J e_{M_j} Jm_k^* J) x( Jm_k J e_{M_j} Jm_k ^*J)\big ) +\sum_{k\neq \ell} \iota_{\rm nor}\big ((Jm_k J e_{M_j} Jm_k^* J )x (Jm_\ell J e_{M_j} Jm_\ell ^*J )\big ).
\end{aligned}
\]
By Remark~\ref{rem:compact}, we have $(Jm_k J e_{M_j} Jm_k^* J)( x-E_{M_j}(x))(Jm_\ell J e_{M_j} Jm_\ell ^*J)\in \B(L^2M)$ is a compact operator from $M$ to $L^2M$ for $k\neq \ell$.
Since $(Jm_k J e_{M_j} Jm_k^* J)E_{M_j}(x)(Jm_\ell J e_{M_j} Jm_\ell ^*J)=0$ if $\ell\neq k$, we have $\sum_{k\neq \ell} q_\K^\perp  \iota_{\rm nor}(Jm_k J e_{M_j} Jm_k^* J x Jm_\ell J e_{M_j} Jm_\ell ^*J) =0$.
Similarly, one checks that $q_\K^\perp\iota _{\rm nor}\big( (Jm_k J e_{M_j} Jm_k^* J) x( Jm_k J e_{M_j} Jm_k ^*J)\big)= q_\K^\perp \iota _{\rm nor}\big( E_{M_j}(x)    (Jm_k J e_{M_j}Jm_k^* J)\big)$.

It then follows from (\ref{equa1}) that $\phi_0$ is a $*$-homomorphism.  
Now we verify that  $\phi_0$ is norm continuous.  

Given $\sum_{i=1}^d m_{k_i} a_i e_{M_j} m_{\ell_i}^*\in A_0$, we may assume that $k_i\neq k_j$ and $\ell_i\neq \ell_j$ if $i\neq j$.
Consider $P_k=  q_\K^\perp \sum_{i=1}^d \iota_{\rm nor}(Jm_{k}Jm_{\ell_i}e_{M_j}m_{\ell_i}^*Jm_{k}^ *J)$
	and $Q_k = q_\K^\perp \sum_{i=1}^d \iota_{\rm nor}(Jm_{k}Jm_{k_i}e_{M_j}m_{k_i}^*Jm_{k}^ *J)$.
We have $P_k$ and $Q_k$ are a projections and $P_k P_r=Q_k Q_r=0$ if $k\neq r$ by Remark~\ref{rem:compact}.
And for the same reason, we have $\iota_{\rm nor}(e_{M_j} m_{\ell_i}^* Jm_k^*J) P_k= q_\K^\perp \iota_{\rm nor}(e_{M_j} m_{\ell_i}^* Jm_k^*J)$
	as well as $\iota_{\rm nor}(e_{M_j} m_{k_i}^* Jm_k^*J) Q_k=q_\K^\perp \iota_{\rm nor}(e_{M_j} m_{k_i}^* Jm_k^*J)$ for each $1\leq i\leq d$.
Let $\cH$ be the Hilbert space where $(\B(L^2M)^\sharp_J)^*$ is represented on.
For $\xi, \eta\in (\cH)_1$, we compute 
\[
\begin{aligned}
|\langle \phi_0(\sum_{i=1}^d m_{k_i} a_i e_{M_j} m_{\ell_i}^*)\xi, \eta\rangle|
&\leq \sum_{k\geq 0} | \sum_{i=1}^d \langle q_\K^\perp \iota_{\rm nor}(e_{M_j} m_{\ell_i}^* Jm_k ^* J)\xi, \iota_{\rm nor} ( Jm_k J m_{k_i}e_{M_j} a_i )^*\eta\rangle|\\
&= \sum_{k\geq 0} | \sum_{i=1}^d \langle  \iota_{\rm nor}(e_{M_j} m_{\ell_i}^* Jm_k ^* J)P_k\xi, \iota_{\rm nor} ( Jm_k J m_{k_i}e_{M_j} a_i )^*Q_k\eta\rangle|\\
&\leq \sum_{k\geq 0} \|\iota_{\rm nor}(Jm_k J(\sum_{i=1}^d  m_{k_i} a_i e_{M_j} m_{\ell_i}^*) Jm_k ^* J)\|\|P_k\xi \| \|Q_k \eta\|\\
&\leq (\sup_{k\in \mathbb{N}}\|m_k\|^2)\|\sum_{i=1}^d  m_{k_i} a_i e_{M_j} m_{\ell_i}^*\|(\sum_{k\geq 0} \|P_k\xi \|^2)^{1/2} (\sum_{k\geq 0} \|Q_k\xi \|^2)^{1/2}\\
&\leq   (\sup_{k\in \mathbb{N}}\|m_k\|^2)\|\sum_{i=1}^d  m_{k_i} a_i e_{M_j} m_{\ell_i}^*\|.
\end{aligned}
\]
This shows that $\phi_0$ is norm continuous as required.

Lastly we show that $\phi_0$ maps into $q_\K^\perp \tilde \bS(M)$.
It suffices to show that $[e_j,\iota_{\rm nor}(Jm_\ell uJ)]=0$ for all $\ell \in \mathbb{N}$ and $u\in \mathcal{U}(M_j)$,
    since $\phi_0(A_0)$ commutes with $\iota_{\rm nor}(JMJ)$.

Without loss of generality, we may assume that $m_0=1$. We compute 
\[
\begin{aligned}
[e_j, \iota_{\rm nor}( Jm_\ell uJ)]&=  q_\K^\perp \big(\sum_{k\geq 0} \iota_{\rm nor}(Jm_kJe_{M_j}Jm_k^*m_\ell uJ)- \iota_{\rm nor}(Jm_\ell um_kJe_{M_j}Jm_k^*J)\big)\\
&=  q_\K^\perp \big(\sum_{k\geq 0}\iota_{\rm nor} (Jm_ke_{M_j}m_k^*m_\ell e_{M_j}ue_{M_j}J)-\iota_{\rm nor}   (Jm_\ell e_{M_j}ue_{M_j}m_ke_{M_j}m_k^*J)\big) \\
&= q_\K^\perp \big( \iota_{\rm nor} (Jm_\ell uJe_{M_j})- \iota_{\rm nor} (Jm_\ell uJe_{M_j})\big)=0.
\end{aligned}
\]
Combining all the above arugments, we may extend $\phi_0: A\to q_\K^\perp \tilde \bS(M)$ to a $*$-homomorphism on $A$, 
    where $A=\overline{A_0}^{\|\cdot\|}$ is a ${\rm C^*}$-algebra. 

The next step is to define the map $\phi$. 
For each $n\geq 0$, set $\phi_n:= \phi_0\circ \psi_n: \langle M, e_{M_j}\rangle \to q_\K^\perp \tilde \bS(M)$,
	which is c.p.\ and subunital by construction.
We may then pick $\phi\in CB(\langle M, e_{M_j}\rangle, q_\K^\perp \tilde \bS(M))$ a weak$^*$ limit point of $\{\phi_n\}_{n\in \mathbb{N}}$, 
	which exists as $q_\K^\perp \tilde \bS(M)$ is a von Neumann algebra.

We claim that 
$$\Ad(q_j)\circ \phi:\langle M, e_{M_j}\rangle \to q_{\K}^\perp q_j \widetilde{\bS}(M) q_{j} $$ 
is an $M$-bimodular u.c.p.\ map, which amounts to showing $\phi(x)=q_{\K}^\perp q_j \iota_{\rm nor}(x)$ for any $x\in M$. 

In fact, for any $x\in M$, we have 
\[
\begin{aligned}
    \phi(x)=&\lim_{n\to\infty} \phi_0\Big(\sum_{0\leq k,\ell\leq n} (m_k  E_{M_j}( m_k^* x m_\ell)e_{M_j}  m_\ell ^*)\Big)\\
    =& q_{\K}^{\perp} \lim_{n\to\infty} \sum_{0\leq k,\ell\leq n} \iota_{\rm nor}( m_k  E_{M_j}( m_k^* x m_\ell) ) e_{j}\iota_{\rm nor}( m_\ell ^*)\\
	=&  q_{\K}^{\perp} \lim_{n\to\infty} \sum_{0\leq k,\ell\leq n} \big ( \iota_{\rm nor} (m_k) e_{j} \iota_{\rm nor}(m_k^*)\big ) 
	\iota_{\rm nor}( x)\big( \iota_{\rm nor}( m_\ell )e_{j} \iota_{\rm nor}(m_\ell^*)\big),\\
\end{aligned}
\]
where the last equation follows from (\ref{equa1}).
Finally, note that $\{p_k\}_{k\geq 0}$ is a family of pairwise orthogonal projections by Remark~\ref{rem:compact} , where
$$p_k:=q_\K^\perp  \iota_{\rm nor} (m_k) e_{j} \iota_{\rm nor}(m_k^*)=q_\K^\perp \sum_{r\geq 0} \iota_{\rm nor}(Jm_r J m_k e_{M_j} m_k^* Jm_r^*J),$$
	and $\sum_{k\geq 0}p_k= \sum_{k,r\geq 0}q_\K^\perp \iota_{\rm nor}(Jm_r J m_k e_{M_j} m_k^* Jm_r^*J)=q_\K^\perp q_j$ by Lemma \ref{lem: identity}. 
Therefore, we conclude that $\phi(x)=q_\K^\perp q_j \iota_{\rm nor}(x)$, as desired.   
\end{proof}

Now consider
	$\nu=\varphi\circ\phi\in \langle M, e_{M_j}\rangle^*$ 
	and notice that $\frac{1}{\varphi(q_\K^\perp q_j)}\nu$ is an $Af$-central state, 
	which is a normal state when restricted to $fMf$.
Let $f_j\in \mathcal Z((Af)'\cap fMf)$ be the support projection of $\nu_{|\mathcal Z((Af)'\cap fMf)}$ and then we have $Af_j$ is amenable relative to $M_j$ inside $M$ \cite[Theorem 2.1]{OzPo10a}.
Apply the same argument for each $i$ with $\varphi(q_\K^\perp q_i)>0$, we then obtain  projections $f_i\in fMf$ (possibly $0$) such that $Af_i$ is amenable relative to $M_i$ inside $M$.

Finally, to show $\vee_{i=1}^n f_i=f$, 
	note that $\varphi(q_i f_i)=\varphi(q_i)$ as 
	$$\varphi(q_i f_i^\perp)=\varphi(q_\K^\perp q_i f_i^\perp)=\varphi(\phi(f_i^\perp))=\nu(f_i^\perp)=0.$$
Consequently we have 
	$$\varphi(\vee_{i=1}^n f_i)\geq \varphi(\vee_{i=1}^n q_i f_i)\geq \varphi(\vee_{i=1}^n q_i)=1,$$ 
	and hence $\vee_{i=1}^n f_i=f$ by the faithfulness of $\varphi_{\mid fMf}$. 
Since $f_i\in \mathcal Z((Af)'\cap fMf)$, we may rearrange these projections so that $\sum_{i=1}^n f_i=f$.
\end{proof}

\section{Proofs of main theorems}

\begin{proof}[Proof of Theorem~\ref{cor:biexact}]
This follows from noticing that the Jones projections $e_{L\Lambda_i}$ pairwise commute, and then applying Theorem~\ref{thm:biexact dichotomy} and Thoerem~\ref{prop:bootstrap}.
\end{proof}

\begin{thm}\label{thm:subalg in free product}
Let $(M_1,\tau_1)$ and $(M_2, \tau_2)$ be such that $M_i\cong L\Gamma_i$ where $\Gamma_i$ are countable exact groups and $M=M_1\ast M_2$ be the tracial free product. 
Let $A\subset M$ be von Neumann subalgebra, then there exists projections $\{p_i\}_{i=1}^3\in \mathcal{Z}(A'\cap M)$  such that $Ap_i\prec_M M_i$ for each $i=1$ and $2$, $Ap_3$ is amenable and $A(\vee_{i=1}^3 p_i)^\perp$ is properly proximal.
\end{thm}

\begin{proof}[Proof of Theorem~\ref{thm:subalg in free product}]
First note that the free products of the exact groups $\Gamma_i$ is  biexact relative to $\{\Gamma_1,\Gamma_2\}$  \cite[Proposition 15.3.12]{BO08} and $[e_{M_1}, e_{M_2}]=0$.
Then by Theorem~\ref{thm:biexact dichotomy}, we may take $f_1$ and $f_2$ from Theorem~\ref{prop:bootstrap} and let $p_i'\in \mathcal Z(Af_i)$ be the maximal projection such that $Ap_i'$ is amenable for each $i=1,2$.
Set $p_i=f_i-p_i'$ for $i=1$ and $2$, and $p_3=p_1'+ p_2'$ and the rest follows from Lemma~\ref{Ioana}.
\end{proof}

\begin{proof}[Proof of Corollary~\ref{cor:inner amen absorption}]
Since $A\subset M$ has no properly proximal direct summand, it follows from Theorem~\ref{prop:bootstrap} and Theorem~\ref{thm:biexact dichotomy} that there exists central projections $f_1$ and $f_2$ in $\cZ(A'\cap M)$ such that $Af_i$ is amenable relative to $M_i$ inside $M$ for each $i$, and $f_1+ f_2=1$.

If $Af_2$ is not amenable, then by Lemma \ref{Ioana} we have that $Af_2\prec_M M_2$.
However, since $A\cap M_1$ is diffuse, we may pick a sequence of trace zero unitaries $\{u_n\}$ in $A\cap M_1$ converging to 0.
One then checks that $\|E_{M_2}(x u_nf_2 y)\|_2\to 0$ for any $x$, $y\in M$, which is a contradiction. 
Therefore $A$ is amenable relative to $M_1$ inside $M$. And then it follows from Theorem~\ref{boutonnet houdayer theorem} that $A\subset M_1$.
\end{proof}

\begin{proof}[Proof of Corollary \ref{Kurosh}]
Note that in the case of $A=L\Gamma_1$, we have $\cZ(A'\cap M)=\C$ and hence Theorem~\ref{thm:subalg in free product} implies that
	either $L\Gamma_1\prec_M L\Lambda_1$ or $L\Gamma_1\prec_M L\Lambda_2\ast \cdots \ast L\Lambda_m$.
The same argument as in \cite[Corollary 8.1]{Drimbe1} deduces the desired result.
\end{proof}



\bibliographystyle{amsalpha}
\bibliography{ref}

\providecommand{\bysame}{\leavevmode\hbox to3em{\hrulefill}\thinspace}
\providecommand{\MR}{\relax\ifhmode\unskip\space\fi MR }
\providecommand{\MRhref}[2]{%
  \href{http://www.ams.org/mathscinet-getitem?mr=#1}{#2}
}
\providecommand{\href}[2]{#2}
\begin{thebibliography}{CFRW10}

\bibitem[BC15]{BC15}
R\'{e}mi Boutonnet and Alessandro Carderi, \emph{Maximal amenable von {N}eumann
  subalgebras arising from maximal amenable subgroups}, Geom. Funct. Anal.
  \textbf{25} (2015), no.~6, 1688--1705. \MR{3432155}

\bibitem[BC17]{boutonnetcarderimaxamen}
R{\'e}mi Boutonnet and Alessandro Carderi, \emph{Maximal amenable subalgebras
  of von neumann algebras associated with hyperbolic groups}, Mathematische
  Annalen \textbf{367} (2017), no.~3, 1199--1216.

\bibitem[BEW19]{BEW19}
Alcides Buss, Siegfried Echterhoff, and Rufus Willett, \emph{Injectivity,
  crossed products, and amenable group actions}, 2019, arXiv:1904.06771.

\bibitem[BH18]{BH18}
R\'{e}mi Boutonnet and Cyril Houdayer, \emph{Amenable absorption in amalgamated
  free product von {N}eumann algebras}, Kyoto J. Math. \textbf{58} (2018),
  no.~3, 583--593. \MR{3843391}

\bibitem[BIP18]{BIP18}
R\'emi Boutonnet, Adrian Ioana, and Jesse Peterson, \emph{Properly proximal
  groups and their von {N}eumann algebras}, 2018, arXiv:1809.01881.

\bibitem[Bla06]{Bl06}
B.~Blackadar, \emph{Operator algebras}, Encyclopaedia of Mathematical Sciences,
  vol. 122, Springer-Verlag, Berlin, 2006, Theory of $C^*$-algebras and von
  Neumann algebras, Operator Algebras and Non-commutative Geometry, III.

\bibitem[BO08]{BO08}
Nathanial~P. Brown and Narutaka Ozawa, \emph{{$C^*$}-algebras and
  finite-dimensional approximations}, Graduate Studies in Mathematics, vol.~88,
  American Mathematical Society, Providence, RI, 2008. \MR{2391387}

\bibitem[BW16]{1c}
Arnaud Brothier and Chenxu Wen, \emph{The cup subalgebra has the absorbing
  amenability property}, Internat. J. Math. \textbf{27} (2016), no.~2, 1650013,
  6. \MR{3464393}

\bibitem[CFRW10]{radialmasa}
Jan Cameron, Junsheng Fang, Mohan Ravichandran, and Stuart White, \emph{The
  radial masa in a free group factor is maximal injective}, J. Lond. Math. Soc.
  (2) \textbf{82} (2010), no.~3, 787--809. \MR{2739068}

\bibitem[CH10]{ChHou}
Ionut Chifan and Cyril Houdayer, \emph{Bass-{S}erre rigidity results in von
  {N}eumann algebras}, Duke Math. J. \textbf{153} (2010), no.~1, 23--54.
  \MR{2641939}

\bibitem[CS04]{ceccherini}
Tullio Ceccherini-Silberstein, \emph{On subfactors with a unitary orthonormal
  basis}, Sovrem. Mat. Prilozh. (2004), no.~22, Algebra i Geom., 102--125.
  \MR{2462073}

\bibitem[Din22]{Din22}
Changying Ding, \emph{First $\ell^2$-betti number and proper proximality},
  2022, In preparation.

\bibitem[DKE22]{DKE21}
Changying Ding and Srivatsav Kunnawalkam~Elayavalli, \emph{Proper proximality
  for various families of groups}, 2022, arXiv: 2107.02917.

\bibitem[DKEP22]{DKEP21}
Changying Ding, Srivatsav Kunnawalkam~Elayavalli, and Jesse Peterson,
  \emph{Properly proximal von {N}eumann algebras}, 2022, arXiv:
  https://arxiv.org/abs/2204.00517.

\bibitem[DP22]{DP22}
Changying Ding and Jesse Peterson, \emph{Biexact von {N}eumann algebras}, 2022,
  In preparation.

\bibitem[Dri22]{Drimbe1}
Daniel Drimbe, \emph{Measure equivalence rigidity via s-malleable
  deformations}, 2022.

\bibitem[ER88]{EfRu88}
Edward~G. Effros and Zhong-Jin Ruan, \emph{Representations of operator
  bimodules and their applications}, J. Operator Theory \textbf{19} (1988),
  no.~1, 137--158.

\bibitem[HJNS21]{hayesetal}
Ben Hayes, David Jekel, Brent Nelson, and Thomas Sinclair, \emph{A random
  matrix approach to absorption in free products}, Int. Math. Res. Not. IMRN
  (2021), no.~3, 1919--1979. \MR{4206601}

\bibitem[Hou14]{houdayergammaabsorption}
Cyril Houdayer, \emph{Gamma stability in free product von neumann algebras},
  Communications in Mathematical Physics \textbf{336} (2014).

\bibitem[HU16]{Houdayer_2016}
Cyril Houdayer and Yoshimichi Ueda, \emph{Rigidity of free product von~neumann
  algebras}, Compositio Mathematica \textbf{152} (2016), no.~12, 2461--2492.

\bibitem[Ioa15]{Ioa15}
Adrian Ioana, \emph{Cartan subalgebras of amalgamated free product {${\rm
  II}_1$} factors}, Ann. Sci. \'{E}c. Norm. Sup\'{e}r. (4) \textbf{48} (2015),
  no.~1, 71--130, With an appendix by Ioana and Stefaan Vaes. \MR{3335839}

\bibitem[Ioa18]{Io18}
\bysame, \emph{Rigidity for von neumann algebras}, Proceedings of the
  {I}nternational {C}ongress of {M}athematicians. {V}olume {II}, 2018,
  pp.~1635--1668.

\bibitem[IPP08]{IPP08}
Adrian Ioana, Jesse Peterson, and Sorin Popa, \emph{Amalgamated free products
  of weakly rigid factors and calculation of their symmetry groups}, Acta Math.
  \textbf{200} (2008), no.~1, 85--153. \MR{2386109}

\bibitem[IPR19]{IsPeRu19}
Ishan Ishan, Jesse Peterson, and Lauren Ruth, \emph{Von {Neumann} equivalence
  and properly proximal groups}, arXiv:1910.08682, 2019.

\bibitem[Jol12]{jolissaint}
Paul Jolissaint, \emph{Examples of mixing subalgebras of von {N}eumann algebras
  and their normalizers}, Bull. Belg. Math. Soc. Simon Stevin \textbf{19}
  (2012), no.~3, 399--413. \MR{3027351}

\bibitem[Mag97]{Ma97}
Bojan Magajna, \emph{Strong operator modules and the {H}aagerup tensor
  product}, Proc. London Math. Soc. (3) \textbf{74} (1997), no.~1, 201--240.

\bibitem[Mag98]{Ma98}
\bysame, \emph{A topology for operator modules over {$W^*$}-algebras}, J.
  Funct. Anal. \textbf{154} (1998), no.~1, 17--41.

\bibitem[Mag05]{Ma05}
\bysame, \emph{Duality and normal parts of operator modules}, J. Funct. Anal.
  \textbf{219} (2005), no.~2, 306--339.

\bibitem[OP10]{OzPo10a}
Narutaka Ozawa and Sorin Popa, \emph{On a class of {${\rm II}_1$} factors with
  at most one {C}artan subalgebra}, Ann. of Math. (2) \textbf{172} (2010),
  no.~1, 713--749. \MR{2680430}

\bibitem[Osi06]{osin}
Denis~V. Osin, \emph{Relatively hyperbolic groups: intrinsic geometry,
  algebraic properties, and algorithmic problems}, Mem. Amer. Math. Soc.
  \textbf{179} (2006), no.~843, vi+100. \MR{2182268}

\bibitem[Oya22]{Koichi}
Koichi Oyakawa, \emph{Bi-exactness of relatively hyperbolic groups}, 2022.

\bibitem[Oza04]{Oza04}
Narutaka Ozawa, \emph{Solid von {N}eumann algebras}, Acta Math. \textbf{192}
  (2004), no.~1, 111--117. \MR{2079600}

\bibitem[Oza06]{Oza06}
\bysame, \emph{A {K}urosh-type theorem for type {$\rm II_1$} factors}, Int.
  Math. Res. Not. (2006), Art. ID 97560, 21. \MR{2211141}

\bibitem[Oza10a]{Oza10}
\bysame, \emph{A comment on free group factors}, Noncommutative harmonic
  analysis with applications to probability {II}, Banach Center Publ., vol.~89,
  Polish Acad. Sci. Inst. Math., Warsaw, 2010, pp.~241--245. \MR{2730894}

\bibitem[Oza10b]{ozawanote}
\bysame, \emph{A comment on free group factors}, Noncommutative harmonic
  analysis with applications to probability {II}, Banach Center Publ., vol.~89,
  Polish Acad. Sci. Inst. Math., Warsaw, 2010, pp.~241--245. \MR{2730894}

\bibitem[Pet09]{Jesse}
Jesse Peterson, \emph{{$L^2$}-rigidity in von {N}eumann algebras}, Invent.
  Math. \textbf{175} (2009), no.~2, 417--433. \MR{2470111}

\bibitem[Pop83]{Popa1983MaximalIS}
Sorin Popa, \emph{Maximal injective subalgebras in factors associated with free
  groups}, Advances in Mathematics \textbf{50} (1983), 27--48.

\bibitem[Pop06]{popainvent}
Sorin Popa, \emph{Strong rigidity of {$\rm II_1$} factors arising from
  malleable actions of {$w$}-rigid groups. {I}}, Invent. Math. \textbf{165}
  (2006), no.~2, 369--408. \MR{2231961}

\bibitem[PP86]{PPbasis}
Mihai Pimsner and Sorin Popa, \emph{Entropy and index for subfactors}, Ann.
  Sci. \'{E}cole Norm. Sup. (4) \textbf{19} (1986), no.~1, 57--106. \MR{860811}

\bibitem[PSW18]{SKC}
Sandeepan Parekh, Koichi Shimada, and Chenxu Wen, \emph{Maximal amenability of
  the generator subalgebra in {$q$}-{G}aussian von {N}eumann algebras}, J.
  Operator Theory \textbf{80} (2018), no.~1, 125--152. \MR{3835452}

\bibitem[Ser80]{Ser80}
Jean-Pierre Serre, \emph{Trees}, Springer-Verlag, Berlin-New York, 1980,
  Translated from the French by John Stillwell. \MR{607504}

\bibitem[Vae12]{vaesinneramenable}
Stefaan Vaes, \emph{An inner amenable group whose von {N}eumann algebra does
  not have property {G}amma}, Acta Math. \textbf{208} (2012), no.~2, 389--394.
  \MR{2931384}

\bibitem[Wen16]{wen}
Chenxu Wen, \emph{Maximal amenability and disjointness for the radial masa}, J.
  Funct. Anal. \textbf{270} (2016), no.~2, 787--801. \MR{3425903}

\end{thebibliography}

\end{document}